\documentclass{amsart}
\usepackage{amssymb,latexsym}
\usepackage{graphicx}
\usepackage[inline]{enumitem}
\usepackage{multicol}
\usepackage{amsmath,amsfonts}  
\usepackage{subcaption}
\usepackage{hyperref}

\graphicspath{ {Images/} }

\newtheorem{thm}{Theorem}
\newtheorem{theorem}{Theorem}[section]
\newtheorem{prop}[theorem]{Proposition}
\newtheorem{lemma}[theorem]{Lemma}

\newtheorem{claim}{Claim}

\usepackage{tikz}
\usepackage{verbatim}
\usepackage{color}
\usetikzlibrary{matrix}

\begin{document}

\title[Classification of Finite-Dimensional Continua]{The Complexity of the Classification Problems of Finite-Dimensional Continua}

\author{Cheng Chang}
\address{School of Liberal Arts\\ Mercy College\\555 Broadway\\ Dobbs Ferry,  NY 10522\\USA}
\email{cchang4@mercy.edu}
\author{Su Gao}
\address{Department of Mathematics\\ University of North Texas\\ 1155 Union Circle \#311430\\  Denton, TX 76203\\ USA}
\email{sgao@unt.edu}

\date{\today}
\thanks{Su Gao's research was supported in part by NSF grant DMS-1800323.}

\begin{abstract} We consider the homeomorphic classification of finite-dimensional continua as well as several related equivalence relations. We show that, when $n\geq 2$, the classification problem of $n$-dimensional continua is strictly more complex than the isomorphism problem of countable graphs. We also obtain results that compare the relative complexity of various equivalence relations.
\end{abstract}

\maketitle \thispagestyle{empty}


\section{Introduction}

In \cite{CG} we determined the exact complexity of the homeomorphic classification problem of all continua, i.e., connected compact metric spaces. In this paper we consider continua that are subspaces of finite-dimensional Euclidean spaces. The framework of our study is the descriptive set theory of equivalence relations, which we briefly review below. The reader could consult \cite{G09} for more details. 

Let $X, Y$ be standard Borel spaces and $E, F$ be equivalence relations on $X, Y$, respectively. We say that $E$ is {\it Borel reducible} to $F$, denoted $E\leq_B F$, if there is a Borel function $\varphi: X\to Y$ such that for all $x, y\in X$, $xEy\iff \varphi(x)F\varphi(y)$. We say that $E$ is {\it strictly Borel reducible} to $F$, denoted $E<_B F$, if $E\leq_B F$ and $F\not\leq_B E$. $E$ is said to be {\it Borel bireducible} with $F$, denoted $E\sim_B F$, if both $E\leq_BF$ and $F\leq_B E$. If $\mathcal{C}$ is a class of equivalence relations and $F\in \mathcal{C}$, we say that $F$ is {\it universal} for $\mathcal{C}$ if for all $E\in \mathcal{C}$, we have $E\leq_B F$. 

Classification problems in mathematics can  often be viewed as equivalence relations on standard Borel spaces. In continuum theory, for instance, let $\mathcal{C}([0,1]^\mathbb{N})$ be the space of all non-empty connected closed subsets of the Hilbert cube $[0,1]^\mathbb{N}$. Then $\mathcal{C}([0,1]^\mathbb{N})$ can be viewed as the space of all continua since every continuum is homeomorphic to a subspace of the Hilbert cube. It is well-known that $\mathcal{C}([0,1]^\mathbb{N})$ is a standard Borel space. Thus the homeomorphic classification problem of all continua becomes an equivalence relation on the standard Borel space $\mathcal{C}([0,1]^\mathbb{N})$. 

The notion of Borel reducibility becomes a way to talk about the relative complexity of classification problems. If $E, F$ are classification problems with $E<_B F$, then $F$ is strictly more complex than $E$. On the other hand, if $E\sim_B F$, then $E$ and $F$ are of the same complexity.

To determine the exact complexity of an equivalence relation we often use a benchmark equivalence relation, i.e., an equivalence relation that is easy to define and which shows up frequently in research. Another important way for an equivalence relation to become a benchmark is for it to be universal in a significant class of equivalence relations. For example, Zielinski in \cite{Z} showed that the homeomorphic classification problem for all compact metric spaces is Borel bireducible with a universal orbit equivalence relation arising from a Borel action of a Polish group. We showed in \cite{CG} that the classification problem of all continua is also Borel bireducible to this equivalence relation. Because the universal orbit equivalence relation is a well-known benchmark, we have thus determined the exact complexity of these classification problems.

The benchmark equivalence relation we use in this paper is the isomorphism relation of all countable graphs, which is also known as the {\it graph isomorphism}. Formally, let $\mathcal{G}$ be the space of all graphs $(V, E)$ with $V=\mathbb{N}$. Then $\mathcal{G}\subseteq 2^{\mathbb{N}\times\mathbb{N}}$ can be shown to be a standard Borel space. The graph isomorphism is thus an equivalence relation on $\mathcal{G}$. It is well-known that the graph isomorphism is Borel bireducible to a universal orbit equivalence relation arising from a Borel action of the infinite permutation group $S_\infty$. Thus the graph isomorphism is sometimes also said to be {\it $S_\infty$-universal} (e.g. \cite{CDM}).

In this paper we will consider the homeomorphic classification problem of all subcontinua of $[0,1]^n$, which we denote by $\mathsf{C}_n$. In comparison, we will also consider the homeomorphic classification problem of all closed subsets of $[0,1]^n$, which we denote by $\mathsf{H}_n$. In addition, we consider the following equivalence relation $\mathsf{R}_n$ among all closed subsets of $[0,1]^n$. If $A, B$ are closed subsets of $[0,1]^n$, then $(A, B)\in \mathsf{R}_n$ iff there is a homeomorphism $f: [0,1]^n\to [0,1]^n$ with $f[A]=B$. 

One easily sees that $\mathsf{C}_1$ has only two equivalence classes. It is a folklore that both $\mathsf{R}_1$ and $\mathsf{H}_1$ are Borel bireducible with the graph isomorphism (we will give a proof later in this paper). When we compare the equivalence relations $\mathsf{C}_n$, $\mathsf{H}_n$ and $\mathsf{R}_n$ in terms of Borel reducibility, it is obvious that $\mathsf{C}_n\leq_B  \mathsf{H}_n$, $\mathsf{H}_n\leq_B \mathsf{H}_{n+1}$, and $\mathsf{C}_n\leq_B \mathsf{C}_{n+1}$. The following results are less obvious.

\begin{thm} The following hold for any $n$:
\begin{enumerate}
\item[\rm (1)] $\mathsf{H}_n\leq_B \mathsf{C}_{n+2}$;
\item[\rm (2)] $\mathsf{R}_n\leq_B \mathsf{C}_{n+2}$;
\item[\rm (3)] $\mathsf{R}_n\leq_B \mathsf{R}_{n+1}$.
\end{enumerate}
\end{thm}

It follows that the graph isomorphism is Borel reducible to all $\mathsf{R}_n$ and $\mathsf{H}_n$. Camerlo, Darji, and Marcone showed in \cite{CDM} that the graph isomorphism is Borel reducible to $\mathsf{C}_2$, and hence to all $\mathsf{C}_n$ for $n\geq 2$. Our main result of the paper is the following.

\begin{thm}\label{main} For any $n\geq 2$, the graph isomorphism is strictly Borel reducible to each of $\mathsf{C}_n, \mathsf{H}_n$, and $\mathsf{R}_n$.
\end{thm}

In particular, Theorem~\ref{main} tells us that it is impossible to assign a countable graph (or any countable structure) as a complete homeomorphic invariant for a finite-dimensional continuum if the dimension is at least $2$.

\section{Preliminaries}

Our standard references for notation and terminology are \cite{K} and \cite{G09}. 

Recall that a {\em Polish space} is a separable, completely metrizable topological space. A {\em standard Borel space} is a pair $(X, \mathfrak{B})$, where $X$ is a set and $\mathfrak{B}$ is a $\sigma$-algebra of subsets of $X$, such that $\mathfrak{B}$ is the $\sigma$-algebra generated by some Polish topology on $X$. If $(X, \mathfrak{B})$ is a standard Borel space we refer to elements of $\mathfrak{B}$ as {\em Borel sets}. As usual, if $(X,\mathfrak{B})$ is a standard Borel space and the collection $\mathfrak{B}$ is clear from the context, we will say that $X$ is a standard Borel space. It is natural to view any Polish space as a standard Borel space.

If $X$ and $Y$ are standard Borel spaces, a function $f: X\to Y$ is {\em Borel (measurable)} if for any Borel $B\subseteq Y$, $f^{-1}(B)\subseteq X$ is Borel. 

Given any Polish space $X$, the {\em Effros Borel space} $\mathcal{F}(X)$ is the space of all non-empty closed subsets of $X$ with the $\sigma$-algebra generated by the sets of the form
\[
\{F\in \mathcal{F}(X): F\cap U\neq \emptyset\},
\]
where $U\subseteq X$ is open. It is a standard Borel space. 

Given any Polish space $X$, let $\mathcal{C}(X)$ be the subspace of $\mathcal{F}(X)$ consisting of all connected compact subsets of $X$. Then $\mathcal{C}(X)$ is again a standard Borel space. 

We can regard $\mathsf{H}_n$ and $\mathsf{R}_n$ to be equivalence relations on $\mathcal{F}([0,1]^n)$ and $\mathsf{C}_n$ an equivalence relation on $\mathcal{C}([0,1]^n)$. 

For our constructions and proofs we will need the following basic notation and terminology in continuum theory. For unexplained notation and terminology our standard reference is \cite{W}. 

Let $X$ be a connected topological space. An element $x\in X$ is a {\em cut-point} of $X$ if $X-\{x\}$ is disconnected. If $x$ is not a cut-point of $X$, it is a \emph{non-cut point} of $X$. Cut-points are preserved by homeomorphisms, but not necessarily by continuous maps.

If $X$ is a topological space and $x, y\in X$, a {\em path} from $x$ to $y$ is a continuous function $f: [0,1]\to X$ such that $f(0)=x$ and $f(1)=y$. When there is no danger of confusion, we also refer to the graph of such an $f$ as a path. Define $x\sim y$ iff there is a path from $x$ to $y$, for any $x, y\in X$. Then $\sim$ is an equivalence relation, and its equivalence classes are the {\em path-components} of $X$. $X$ is {\em path-connected} if it has only one path-component, or equivalently, if there is a path from $x$ to $y$ for any $x, y\in X$. 

Let $X$ be a path-connected space. We call an element $x\in X$ a {\em path-cut-point} if $X-\{x\}$ is no longer path-connected. Note that path-cut-points are also preserved by homeomorphisms.

\section{Comparing $\mathsf{C}_n$, $\mathsf{H}_n$ and $\mathsf{R}_n$}

We establish in this section the results comparing various homeomorphism problems in terms of Borel reducibility. We will use two constructions in \cite{CG} and \cite{Z} for coding a closed subset (or a sequence of closed subsets) of a compact metric space into the homeomorphism type of a continuum. We briefly describe these two constructions first.

\subsection{The construction of $I(X,A)$}
Let $X$ be a compact metric space and $A\subseteq X$ be a closed subspace containing all isolated points of $X$. Let $\mathcal{D}_{X,A}$ be the collection of $D\subseteq X\times (0,1]$ which is a nonempty set of isolated points so that $\overline{D}-D=A\times\{0\}$. If $D\in \mathcal{D}_{X,A}$ and $A\neq\emptyset$, then the set $D$ is necessarily countably infinite. For any $D\in\mathcal{D}_{X,A}$ let $I(X, A; D)=X\times\{0\}\cup D$. Being a closed subspace of $X\times [0,1]$, $I(X, A; D)$ is still a compact metric space. From \cite{CG} and \cite{Z}, we know that $\mathcal{D}_{X,A}$ is not empty, and that all the $I(X,A; D)$ are homeomorphic as $D\in\mathcal{D}_{X,A}$ varies. Thus, we simply write $I(X,A)$ for any $I(X,A; D)$ for $D\in \mathcal{D}_{X,A}$. If $A$ is empty, we let $I(X, A)=I(X, A; D)$ where $D$ is a singleton.

It now follows that $I(X,A)$ is a coding space for the homeomorphism type of pairs $(X, A)$ where $X$ is a compact metric space and $A\subseteq X$ is a closed subspace.

\begin{prop}[\cite{CG}] \label{thm:Z}
Let $X, Y$ be compact metric spaces, and $A\subseteq X$ and $B\subseteq Y$ be closed subspaces containing all isolated points of $X$ and $Y$, respectively. Then the following are equivalent:
\begin{enumerate}
\item[\rm (i)] $(X, A)\cong (Y, B)$, i.e., there is a homeomorphism $f: X\to Y$ with $f[A]=B$.
\item[\rm (ii)] $I(X, A)$ and $I(Y, B)$ are homeomorphic.
\end{enumerate}
\end{prop}

\subsection{The construction of $J(X,A)$}
Let $X$ be a compact metric space. We define the {\em fan space} $F_X$ of $X$ as the quotient of $X\times[0,1]$ by the equivalence relation $\sim$ defined as
$$ (x, s)\sim (y, t)\iff (x, s)=(y, t) \mbox{ or } s=t=1. $$
The point $[(x,1)]_\sim$ in $F_X$ is a distinguished point; we denote it by $\alpha_X$ and call it the {\em apex}. $X$ can be viewed, again in a canonical way, as a subspace of $F_X$.

$F_X$ is obviously compact. We note that it can be given a canonical metric:
\begin{equation*}\label{E:1} d_F((x,s), (y,t))=2|s-t|+(1-\max\{s,t\})\rho(x,y), 
\end{equation*}
where $\rho<1$ is a compatible metric on $X$. $F_X$ is also clearly a path-connected space: for every point $(x,s)$ there is a canonical path $P$ from $(x, s)$ to $\alpha_X$, namely, 
$$ P(\tau)=(x, s+\tau(1-s)) \mbox{ for $\tau\in[0,1]$}. $$
Therefore $F_X$ is a path-connected continuum. 

Next we code pairs $(X, A)$. Given a compact metric space $X$ and a closed subspace $A\subseteq X$, define $F(X, A)$ as a subspace of the fan space $F_X$:
$$ F(X, A)=\{ [(x,s)]_\sim\in F_X\,:\, s=0 \mbox{ or } x\in A\}. $$
Alternatively, we consider the equivalence relation $\sim$ defined above, restricted to the space 
$$ (X\times\{0\})\cup (A\times [0,1]). $$
$F(X,A)$ is the again the quotient space given by $\sim$.

There is again a canonical homeomorphic copy of $X$ in $F(X, A)$, namely $X\times\{0\}$, and a canonical homeomorphic copy of $F_A$ in $F(X, A)$. It is easy to see that if $X$ is (path-)connected, then so is $F(X,A)$. 

The next coding space $J(X,A)$ is based on the space $I(X, A)$. Write $I(X, A)=X\cup D$, where $D$ is the set of all isolated points in $I(X, A)$. Note that $\overline{D}=D\cup A$. We define $$J(X, A)=F(I(X, A), \overline{D}).$$

\begin{prop}[\cite{CG}]\label{thm:J} Let $X, Y$ be continua without cut-points and $A, B$ be closed subspaces of $X, Y$ respectively. Then the following are equivalent:
\begin{enumerate}
\item[\rm (i)] $(X, A)\cong (Y, B)$.
\item[\rm (ii)] $J(X, A)$ and $J(Y, B)$ are homeomorphic.
\end{enumerate}
\end{prop}

\subsection{Comparing $\mathsf{C}_{n}$ and $\mathsf{H}_{n}$}
In this subsection we compare the complexities of $\mathsf{C}_{n}$ and $\mathsf{H}_{n}$. It is obvious that $\mathsf{C}_n\leq_B  \mathsf{H}_n$, $\mathsf{H}_n\leq_B \mathsf{H}_{n+1}$, and $\mathsf{C}_n\leq_B \mathsf{C}_{n+1}$. Our objective is to show that $\mathsf{H}_n\leq_B \mathsf{C}_{n+2}$ for all $n$. These results can be summarized in the following diagram, where a Borel reducibility claim $E\leq_B F$ is represented by an arrow $E\to F$:
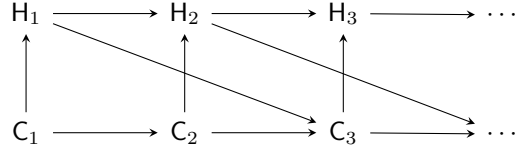
\begin{figure}[h]
\centering
\begin{tikzpicture}
  \matrix (m) [matrix of math nodes,row sep=3em,column sep=4em,minimum width=2em]
  {
     \mathsf{H}_1 & \mathsf{H}_2 & \mathsf{H}_3 & \cdots \\
     \mathsf{C}_1 & \mathsf{C}_2& \mathsf{C}_3& \cdots \\};
  \path[-stealth]
  (m-1-1) edge (m-1-2)
          edge     (m-2-3)
          
  (m-1-2) edge (m-1-3)
            edge (m-2-4)
    (m-1-3) edge (m-1-4)
           
    (m-2-1) edge (m-1-1)
            edge (m-2-2)
    (m-2-2) edge (m-1-2)
            edge (m-2-3)
    (m-2-3) edge (m-1-3)
            edge (m-2-4);
\end{tikzpicture}
\caption{Reductions between $\mathsf{H}_{n}$ and $\mathsf{C}_{n}$. }
\label{Reduction1}
\end{figure}

\newpage
\begin{theorem}\label{HC}
$\mathsf{H}_{n}\leq_{B} \mathsf{C}_{n+2}$ for all $n\geq 1$.
\end{theorem}

The rest of this subsection is devoted to a proof of Theorem~\ref{HC}.

Given any non-empty closed subset $A\subseteq [0,1]^{n}$, consider 
\begin{equation*}
\tilde{A}:=J(A, A)=F_{I(A,A)}.
\end{equation*}
Arbitrarily fix a countable set $D_A\subseteq A\times (0,1]\subseteq [0,1]^n\times(0,1]$ so that $\overline{D_A}-D_A=A\times \{0\}$.  Then $I(A,A)=\overline{D_A}=(A\times\{0\})\cup D_A$. For notational simplicity, we denote the apex of $F_{I(A,A)}$ by $a^*$.

Note that $\tilde{A}$ is a quotient space of $[0,1]^{n+2}$. In the next lemma, we show that it can be embeded as a subspace of $[0,1]^{n+2}$.

\begin{lemma}\label{R:1}
$\tilde{A}$ is homeomorphic to a subspace of $[0,1]^{n+2}$.
\end{lemma}
\begin{proof}
We construct a $\tilde{A}'\subseteq [0,1]^{n+2}$: first embed $I(A,A)=(A\times\{0\})\cup D_{A}$ into $[0,1]^{n+2}$ as $(A\times\{(0,0)\})\cup (D_{A}\times\{0\})$ (called the ``floor'' points); then add an arbitrary point $a'^{\ast} \in [0,1]^{n+1}\times (0,1]$, and connect all the ``floor'' points to $a'^{\ast}$ by straight lines. The set $\tilde{A}'$ is obviously a subset of $[0,1]^{n+2}$, and all the points in $\tilde{A}'$ can be uniquely written as
\[
(1-\lambda ){x}+\lambda a'^{\ast}
\]
for some ${ x}\in I(A,A)\times\{0\}$ and $\lambda\in [0,1]\}$.

Define $\pi: \tilde{A} \rightarrow \tilde{A}'$ by $\pi(x,\lambda)=(1-\lambda) x+\lambda a'^{\ast}$ for $x\in \overline{D}, \lambda\in [0,1]$. Then $\pi$ is a continuous bijection, and thus a homeomorphism.
\end{proof}


Next we state a topological property that separates points of $D_{A}\times\{0\}$ from the other points in $\tilde{A}$.

\begin{lemma}\label{3L1}
Let $p\in \tilde{A}$. Then $p\in D_A\times\{0\}$ iff the following topological property holds for $p$:
\begin{quote}
$p$ is a non-cut point, and for all open neighborhood $V$ of $p$, there exists an open subset $U\subseteq V$ such that ${p}\in U$ and $U$ is path-connected.
\end{quote}
\end{lemma}
\begin{proof}
Note that all the points in $D_{A}\times\{0\}$ are non-cut points. In fact, if ${(x,0)}\in D_{A}\times\{0\}$, then $\tilde{A}-\{(x,0)\}$ is still path-connected, since all points in $\tilde{A}-\{(x,0)\}$ are path-connected to ${a}^{\ast}$. To show the second part of the property for $p=(x,0)\in D_A\times\{0\}$, fix an arbitrary open neighborhood of $(x,0)$, say $V$. Since $x\in D_{A}$ is an isolated point in the space $\overline{D_{A}}=I(A,A)$, there exists some $\epsilon>0$ so that the open set $U:=\{x\}\times [0,\epsilon)\subseteq V$. $U$ is clearly path-connected.

All the points in $\{({x},r): {x}\in D_{A}, r\in (0,1]\}$ are cut-points, so they do not satisfy the displayed property. 

Finally, for the rest of points $({x},r)\in \tilde{A}$, where ${x}\in A\times\{0\}, r\in[0,1]$, there is a sequence of points $\{x_{i}\}_{i\in\mathbb{N}}$ from $D_{A}$ converging to $x$. Then, for every open neighborhood $M$ of $x$ and $\epsilon\in (0,1)$, the basic open set $V:=M\times[0,\epsilon)$ is not connected, as $V$ contains infinitely many disjoint components $\{x_{i}\}\times [0,\epsilon)$ for some $i\in\mathbb{N}$.
\end{proof}

We are now ready to prove Theorem~\ref{HC}. Suppose $A,B$ are non-empty closed subsets of $[0,1]^{n}$, and $\tilde{A},\tilde{B}$ are constructed as above, with $a^*$ and $b^*$ as their respective apexes. Moreover, assume that $\tilde{f}:\tilde{A}\rightarrow \tilde{B}$ is a homeomorphism. By Lemma \ref{3L1}, we have
\[
\tilde{f}(D_{A}\times\{0\})=D_{B}\times\{0\},
\]
hence $\tilde{f}(A\times\{0\}^{2})=B\times\{0\}^{2}$. Therefore, $A,B$ are homeomorphic to each other.

On the other hand, suppose $f:A\rightarrow B$ is a homeomorphism. With the same argument as in the proof of Proposition \ref{thm:Z}, we can extend $f$ into a homeomorphism $f':\overline{D_{A}}\rightarrow \overline{D_{B}}$ such that $f'\restriction_{A\times\{0\}}=f.$ Then  we can extend $f'$ further to $\tilde{f}$ by sending ${a}^{\ast}$ to ${b}^{\ast}$, and $({x},\lambda)\in \overline{D_{A}}\times [0,1)$ to $(f'({x}),\lambda)\in \overline{D_{B}}\times [0,1)$. $\tilde{f}: \tilde{A}\rightarrow \tilde{B}$ is clearly one-to-one, onto and continuous. Since both $\tilde{A}$ and $\tilde{B}$ are compact metric spaces, the continuity of $\tilde{f}$ implies homeomorphism.

Thus we have shown that $A, B$ are homeomorphic iff $\tilde{A}, \tilde{B}$ are homeomorphic. It is straightforward to verify that $A\mapsto \tilde{A}$ as a map from $\mathcal{F}([0,1]^n)$ to $\mathcal{C}([0,1]^{n+2})$ is Borel. Thus $A\mapsto \tilde{A}$ witnesses that $\mathsf{H}_n\leq_B \mathsf{C}_{n+2}$.

\subsection{Comparing $\mathsf{C}_n$ and $\mathsf{R}_{n}$} 
In this subsection we prove $\mathsf{R}_n\leq_B \mathsf{C}_{n+2}$ for all $n$. 
Since $[0,1]^{n}$ is a continua without cut-points for all $n\geq 2$, a direct application of Proposition \ref{thm:J} gives that for all $n\geq 2$ and closed subsets $A,B\subseteq[0,1]^{n}$, we have
$$\begin{array}{rcl}
(A, B)\in \mathsf{R}_n &\iff & ([0,1]^n, A)\cong ([0,1]^n, B)\\ 
&\iff& J([0,1]^{n}, A), J([0,1]^{n},B) \mbox{ are homeomorphic}.
\end{array}
$$
Similar to Lemma~\ref{R:1}, the path-connected spaces $J([0,1]^{n}, A), J([0,1]^{n},B)$ can be embedded as subspaces of $[0,1]^{n+2}$. Therefore, we have $\mathsf{R}_{n}\leq_{B} \mathsf{C}_{n+2}$ for all $n\geq 2$. Now the only case left is when $n=1$, which we address below. 

\begin{theorem}\label{RC} $\mathsf{R}_{1}\leq_{B} \mathsf{C}_{3}$.
\end{theorem}

The rest of this subsection is devoted to a proof of Theorem~\ref{RC}. 
We show again that for non-empty closed subsets $A,B\subseteq [0,1]$, $A, B$ are homeomorphic iff $J([0,1],A),J([0,1],B)$ are homeomorphic. 
The proof of the forward implication is identical to the proof of Proposition \ref{thm:J} (and is straightforward and easy anyway). We only consider the other direction.

Suppose $\tilde{f}: J([0,1],A)\rightarrow J([0,1],B)$ is a homeomorphism. We verify that
\[
\tilde{f}([0,1]\times\{0\}^{2})=[0,1]\times\{0\}^{2},
\]
and
\[
\tilde{f}(A\times\{0\}^{2})=B\times\{0\}^{2}.
\]

Let $a^\ast$ and $b^\ast$ be the apexes of $J([0,1],A)$ and $J([0,1],B)$ resepectively. We first identify a unique topological property for $a^\ast$.
\begin{lemma}
In $J([0,1],A)$, $a^{\ast}$ is the unique cut-point such that $J([0,1],A)-\{a^{\ast}\}$ has infinitely many path-components.
\end{lemma}
\begin{proof}
It is easy to see that $a^{\ast}$ is a cut-point such that $J([0,1],A)-\{a^{\ast}\}$ has infinitely many path-components. In fact, for each $x\in D_A$, $\{x\}\times [0,1)$ is a path-component in $J([0,1],A)-\{a^{\ast}\}$. To see that other points do not satisfy this topological property, we consider them case by case:
\begin{itemize}
\item For all $({x},\lambda)\in J([0,1],A)$, where ${x}\in D_{A}$ and $\lambda\in (0,1)$, $J([0,1],A)-\{({x},\lambda)\}$ has exactly two path-components. 
\item For all $({x},0) \in J([0,1],A)$, where ${x}\in D_{A}$, we have that $(x, 0)$ is a non-cut point.
\item For all $(a,0,0)\in J([0,1],A)$, where $a\geq \max\{A\}$ or $a\leq \min\{A\}$, $J([0,1],A)-\{(a,0,0)\}$ has at most three path-components.
\item For all $(a,0,\lambda)\in J([0,1],A)$, where $\min A< a<\max A$ and $\lambda<1$, we have that $(a,0,\lambda)$ is a non-cut point.
\end{itemize}
\end{proof}

A similar argument show that $b^*$ is the unique cut-point in $J([0,1],B)$ such that $J([0,1],B)-\{b^*\}$ has infinitely many path-components. Thus $\tilde{f}$ sends $a^{\ast}$ to $b^{\ast}$. If we remove $a^{\ast},b^{\ast}$ from their respective spaces, then $\tilde{f}$ sends each path-component in the domain to some path-component in the codomain. 

\begin{lemma}
Assume $A\neq\{0\}$ and $A\neq \{1\}$. Then in the space $J([0,1],A)-\{a^{\ast}\}$, there are two non-homeomorphic types of path-components:
\begin{enumerate}
\item[\rm (i)] $\{x\}\times [0,1)$, where $x\in D_A$;
\item[\rm (ii)] $({[0,1]}\times\{0\}^2)\cup (A\times\{0\}\times [0,1))$.
\end{enumerate}
\end{lemma}
\begin{proof}
For each ${x}\in D_{A}$, $\Lambda_x=\{x\}\times [0,1)$ is a path-component of $J([0,1],A)-\{a^*\}$. Each of these components satisfies both of the following topological properties:
\begin{itemize}
\item There is a unique non-cut point in $\Lambda_x$, namely $({x},0)$;
\item For every cut-point $p\in \Lambda_x$, $\Lambda_x-\{p\}$ has exactly two path-components.
\end{itemize}
Now  $\Delta=({[0,1]}\times\{0\}^2)\cup (A\times\{0\}\times [0,1))$ is also a path-component. If $A$ contains an element $a\in (0,1)$, then $(a, 0, 0)$ is a cut-point of $\Delta$ so that $\Delta-\{(a,0,0)\}$ has at least three path-components. If $A$ does not contain any element in $(0,1)$, then $A=\{0,1\}$ by our assumptions that $A$ is non-empty and yet $A\neq \{0\}$ and $A\neq\{1\}$. In this case, $\Delta$ has no non-cut points.
\end{proof}



For the rest of the proof, we assume without loss of generality that $A, B\neq\{0\},\{1\}$. Then $\tilde{f}$ sends each path-component $\{x\}\times [0,1), {x}\in D_A$ to some $\{y\}\times [0,1), {y}\in D_B$, and sends the path-component $({[0,1]}\times\{0\}^2)\cup (A\times\{0\}\times [0,1))$ to $({[0,1]}\times\{0\}^2)\cup (B\times\{0\}\times [0,1))$. Since $({x},0)$ is the unique non-cut point in  $\{x\}\times [0,1)$ for all ${x}\in D_{A}$, and similarly, $({y},0)$ is the unique non-cut point in $\{y\}\times [0,1)$ for all ${y}\in D_{B}$, we have
\[
\tilde{f}(D_{A}\times\{0\})=D_{B}\times\{0\}.
\]
Hence, we also have $\tilde{f}(\overline{D_{A}}\times\{0\})=\overline{D_{B}}\times\{0\}$, which implies that $\tilde{f}(A\times\{0\}^{2})=B\times\{0\}^{2}$.

We still need to show that $\tilde{f}([0,1]\times\{0\}^{2})=[0,1]\times\{0\}^{2}$. Consider the spaces $J([0,1], A)-(\overline{D_{A}}\times\{0\})$ and $J([0,1], B)-(\overline{D_{B}}\times\{0\})$. $\tilde{f}$ must send the component containing $a^{\ast}$ to the component containing $b^{\ast}$, i.e.
\[
\tilde{f}(\overline{D_{A}}\times (0,1])= \overline{D_{B}}\times (0,1].
\]
Thus, we have shown $\tilde{f}([0,1]\times\{0\}^{2})=[0,1]\times\{0\}^{2}$.

\subsection{Comparing $\mathsf{R}_n$ and $\mathsf{R}_{n+1}$}

In this subsection we compare the complexities among $\mathsf{R}_{n}$ for $n\geq 1$. We will use the well-known fact that for all $n\geq 1$,  if $f:[0,1]^n\to [0,1]^n$ is a homemorphism and $B=\partial [0,1]^n$ is the set of all boundary points, then $f[B]=B$.

\begin{theorem} $\mathsf{R}_n\leq_B \mathsf{R}_{n+1}$ for all $n\geq 1$.
\end{theorem}
\begin{proof}
For a closed $A\subseteq [0,1]^n$, we define $\widehat{A}\subseteq [0,1]^{n+1}$ by first embedding a rescaled copy of $[0,1]^{n}$ on the boundary of $[0,1]^{n+1}$ and then forming a cylinder set off the rescaled copy of $A$:
\[
\widehat{A}:=[\frac{1}{3},\frac{2}{3}]^n\times\{0\}\cup \frac{1}{3}(A+\vec{1})\times [0,\frac{1}{3}],
\]
where 
\[
\frac{1}{3}(A+\vec{1})=\{(\frac{1}{3}a_0+\frac{1}{3},\dots,\frac{1}{3}a_{n-1}+\frac{1}{3}): (a_0,a_1,\dots,a_{n-1})\in A\}.
\]

We verify that $(A, B)\in \mathsf{R}_n$ iff $(\widehat{A}, \widehat{B})\in \mathsf{R}_{n+1}$.
First assume $\widehat{f}:[0,1]^{n+1}\to [0,1]^{n+1}$ is a homeomorphism such that $\widehat{f}[\widehat{A}]=\widehat{B}$. Since $\widehat{f}$ maps the boundary of $[0,1]^{n+1}$ onto itself, and note that $\widehat{A}\cap \partial[0,1]^{n+1}=[\frac{1}{3},\frac{2}{3}]^n\times\{0\}$, $\widehat{f}$ maps $[\frac{1}{3}, \frac{2}{3}]^n\times\{0\}$ onto itself. Thus $\widehat{f}$ induces a homeomorphism $f:[0,1]^n\to[0,1]^n$. More specifically, for any $x\in [0,1]^{n}$, $f(x)=\widehat{f}(\frac{1}{3}(x+\vec{1}),0)$. Meanwhile, we have 
\[
\widehat{f}\displaystyle\left[\frac{1}{3}(A+\vec{1})\times (0,\frac{1}{3}]\right]=\frac{1}{3}(B+\vec{1})\times (0,\frac{1}{3}]
\]
as these are the interior points of $[0,1]^{n+1}$ in $\widehat{A}$ and $\widehat{B}$, respectively. By taking closures, we get, 
\[
\widehat{f}\displaystyle\left[\frac{1}{3}(A+\vec{1})\times\{0\}\right]=\frac{1}{3}(B+\vec{1})\times\{0\}.
\]
Therefore, $f[A]=B$.

Conversely, let $f:[0,1]^n\to [0,1]^n$ with $f[A]=B$. It is enough to define an autohomeomorphism $f'$ on $[0,1]^n$ such that $f'[[\frac{1}{3},\frac{2}{3}]^{n}]=[\frac{1}{3},\frac{2}{3}]^{n}$ and $f'[\frac{1}{3}(A+\vec{1})]=\frac{1}{3}(B+\vec{1})$. Assuming such an $f'$ is defined, then let $\widehat{f}(x, t):=(f'(x),t)$ for all $x\in [0,1]^n$ and $t\in [0,1]$, and $\widehat{f}$ would be an autohomeomorphism of $[0,1]^{n+1}$ with $\widehat{f}[\widehat{A}]=\widehat{B}$. 

Consider the case when $n=1$, whereas there are two cases depending on the orientation of $f$. If $f$ is order-preserving, then define
\[
f'(x)=\begin{cases}\frac{1}{3}[\pi(3x-1)+1], &\mbox{ if $x\in [\frac{1}{3},\frac{2}{3}]$}, \\ x, &\mbox{ if $x\in [0,\frac{1}{3})\cup (\frac{2}{3},1]$.}\end{cases}
\]
If $\pi$ is order-reversing, then let
\[
f'(x)=\begin{cases}\frac{1}{3}[\pi(3x-1)+1], & \mbox{ if $x\in [\frac{1}{3},\frac{2}{3}]$}, \\ 1-x, &\mbox{ if $x\in [0,\frac{1}{3})\cup (\frac{2}{3},1]$. }\end{cases}
\]

For $n\geq 2$, we define $f'$ in two steps. In the first step, let $\phi(x)=\frac{1}{3}(f(3x-\vec{1})+\vec{1})$. Then $\phi$ is an autohomeomorphism of $[\frac{1}{3},\frac{2}{3}]^n$ with $\phi[\frac{1}{3}(A+\vec{1})]=\frac{1}{3}(B+\vec{1})$. It remains to extend $\phi$ to an autohomeomorphism $f'$ of $[0,1]^n$ such that $f'|_{[\frac{1}{3},\frac{2}{3}]^n}=\phi$. 

By recentering and rescaling, our problem is now topologically equivalent to that of extending a given autohomeomorphism on $$B_{1/3}:=\{(x_{0},\dots,x_{n-1})\in \mathbb{R}^{n}:||(x_{0},\dots,x_{n-1})||\leq \frac{1}{3}\}$$ to an autohomeomorphism on $$B_1:=\{(x_{0},\dots,x_{n-1})\in \mathbb{R}^{n}: ||(x_{0},\dots,x_{n-1})||\leq 1\}.$$ At this point we switch to spherical coordinates. Thus
$$B_1=\{(r, \alpha_1,\dots,\alpha_{n-1}): r\in [0,1], \alpha_1,\dots, \alpha_{n-2}\in [0,\pi], \alpha_{n-1}\in [0, 2\pi)\}.$$ The given autohomeomorphism $\phi$ on $B_{1/3}$ must send boundary points to boundary points, that is, for all $\alpha_{1},\dots,\alpha_{n-1}$,
\[
\phi(\frac{1}{3},\alpha_{1},\dots,\alpha_{n-1})=(\frac{1}{3},\alpha_{1}',\dots,\alpha_{n-1}')
\]
for some $\alpha_{1}',\dots,\alpha_{n-1}'$. Let $\pi$ denote the projection map $\pi(r,\alpha_{1},\dots,\alpha_{n-1})=(\alpha_{1},\dots,\alpha_{n-1})$. Now we can define $f'$ as
\[
f'(r,\alpha_{1},\dots,\alpha_{n-1})=\begin{cases}\phi(r,\alpha_{1},\dots,\alpha_{n-1}), & \mbox{ if $r\leq \frac{1}{3}$},\\ (r,\pi\circ \phi(\frac{1}{3},\alpha_{1},\dots,\alpha_{n-1})), & \mbox{ if $r>\frac{1}{3}$.}\end{cases}
\]
$f'$ is clearly a continuous bijection on $B_1$, and thus a homeomorphism.
\end{proof}

The following diagram summarizes our results in the last two subsections regarding $\mathsf{C}_n$ and $\mathsf{R}_n$:

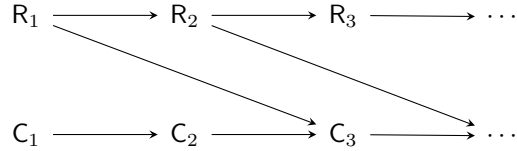
\begin{figure}[h]
\centering
\begin{tikzpicture}
  \matrix (m) [matrix of math nodes,row sep=3em,column sep=4em,minimum width=2em]
  {
     \mathsf{R}_1 & \mathsf{R}_2 & \mathsf{R}_3 & \cdots \\
     \mathsf{C}_1 & \mathsf{C}_2& \mathsf{C}_3& \cdots \\};
  \path[-stealth]
  (m-1-1) edge (m-1-2)
          edge     (m-2-3)
          
  (m-1-2) edge (m-1-3)
            edge (m-2-4)
    (m-1-3) edge (m-1-4)
           
    (m-2-1) 
            edge (m-2-2)
    (m-2-2) 
            edge (m-2-3)
    (m-2-3) 
            edge (m-2-4);
\end{tikzpicture}
\caption{Reductions between $\mathsf{C}_{n}$ and $\mathsf{R}_{n}$.}
\label{Reduction2}
\end{figure}

\section{The Graph Isomorphism and the complexity of $\mathsf{C}_n$, $\mathsf{H}_n$ and $\mathsf{R}_n$}

\subsection{Comparing the graph isomorphism to $\mathsf{H}_1$ and $\mathsf{R}_1$} The graph isomorphism is a benchmark equivalence relation that arises often in the study of classification problems in mathematics, in particular in topology. For example, in \cite{CaG} it was shown that the homeomorphic classification of all zero-dimensional compact metric spaces is Borel bireducible with the graph isomorphism. In fact, the proof shows that the graph isomorphism is in particular reducible to the homeomorphism relation of the closed zero-dimensional subspaces of $[0,1]$. Thus it follows that the graph isomorphism is Borel reducible to $\mathsf{H}_1$. Another example is the result from \cite{CDM} that the graph isomorphism is Borel reducible to the homeomorphism relation of 2-dimensional dendrites. It follows that the graph isomorphism is Borel reducible to $\mathsf{C}_2$.

The following theorem combines results of Friedman and Stanley \cite{FS} and Becker and Kechris \cite{BK}, and further justifies the ubiquity of the graph isomorphism and its status as a benchmark equivalence relation.

\begin{theorem} \label{gi} The following equivalence relations are Borel bireducible with each other:
\begin{enumerate}
\item[\rm (i)] The graph isomorphism, i.e., the isomorphism relation of all countable graphs;
\item[\rm (ii)] The isomorphism relation of all countable linear orderings;
\item[\rm (iii)] The isomorphism relation of all countable $L$-structures, where $L$ is any countable language with at least one $n$-ary relation symbol where $n\geq 2$;
\item[\rm (iv)] A universal equivalence relation for the class of all isomorphism relations of countable $L$-structures, where $L$ varies over all countable languages;
\item[\rm (v)] A universal equivalence relation for the class of all orbit equivalence relations that arise from a Borel action of the infinite permutation group $S_\infty$.
\end{enumerate}
\end{theorem}

For unexplained terminology we refer the reader to \cite{G09}. 

When an equivalence relation or a classification problem is Borel reducible to the graph isomorphism, it means that one can assign a countable graph, a kind of countable structure, as a complete invariant for the equivalence classes. Conversely, if an equivalence relation is classifiable by any kind of countable structures, then by (iv) it can also be classified by countable graphs.

That $\mathsf{H}_1$ and $\mathsf{R}_1$ are Borel bireducible with the graph isomorphism is essentially folklore. For example, in Hjorth \cite{H} the fact that $\mathsf{R}_1$ is Borel reducible to the graph isomorphism is left as an exercise, Exercise 4.13. Here we sketch some proofs for the convenience of the reader. 

\begin{theorem} Both $\mathsf{H}_1$ and $\mathsf{R}_1$ are Borel bireducible with the graph isomorphism.
\end{theorem}

\begin{proof} A Borel reduction from the graph isomorphism to $\mathsf{H}_1$ was given in \cite{CaG}, where it was shown that the graph isomorphism is Borel reducible to the hemeomorphism relation of closed zero-dimensional subsets of $[0,1]$. Here we sketch a proof that $\mathsf{H}_1$ is Borel reducible to the graph isomorphism. In fact, we define a special kind of countable structure and show that $\mathsf{H}_1$ can be classified by these countable structures. Then it follows from Theorem~\ref{gi} that $\mathsf{H}_1$ is Borel reducible to the graph isomorphism.

Given a closed $A\subseteq [0,1]$, we consider its connected components. Note that each connected component of $A$ is either a singleton or a closed interval (of postive length). Since each closed interval contains an open interval, there can be only countably many connected components of $A$ that are intervals. Let $P_A$ be the set of all connected components of $A$ that are closed intervals. Then $P_A$ is a countable set. Let $Q_A$ be the set of all clopen subsets of $A$. Then $Q_A$ is a countable Boolean algebra. Let
$$ \mathfrak{S}_A=(Q_A, P_A, \subseteq) $$
where $\subseteq$ is the relation between an element of $P_A$ and an element of $Q_A$. Then $\mathfrak{S}_A$ is a countable structure encoding $A$. 

More formally, let $L$ be the language
$$ \{ Q, P, \cup, \cap, {\ }^c, \emptyset, I, \subseteq\} $$
where $Q$ and $P$ are unary relation symbols, $\cup, \cap, {\ }^c, \emptyset, I$ are symbols to express that $Q$ is a Boolean algebra, and $\subseteq$ is a relation symbol. In order for the class of $L$-structures to form a standard Borel space, we consider the following axioms in addition to those describing that $Q$ is a Boolean algebra:
\begin{itemize}
\item $\forall x\, (Q(x)\vee P(x))\wedge \neg(Q(x)\wedge P(x))$
\item $\forall x, y\ (x\subseteq y \longrightarrow P(x)\wedge Q(y))$
\end{itemize}

We claim that closed subsets $A, B\subseteq [0,1]$ are homeomorphic iff $\mathfrak{S}_A, \mathfrak{S}_B$ are isomorphic. First, if $A, B$ are homeomorphic, then the homeomorphism gives rise to an isomorphism between $Q_A$ and $Q_B$, which also sends $P_A$ to $P_B$ and preserves the relation $\subseteq$. Thus there is an isomorphism between $\mathfrak{S}_A$ and $\mathfrak{S}_B$. Conversely, suppose there is an isomorphism $\varphi$ between $\mathfrak{S}_A$ and $\mathfrak{S}_B$. Then $\varphi$ gives a bijection between $P_A$ and $P_B$, as well as a bijection between $Q_A$ and $Q_B$. By the Stone duality, the bijection between $Q_A$ and $Q_B$ gives rise to a bijection $\psi$ between the dual space of $Q_A$ and the dual space of $Q_B$. These dual spaces correspond to the connected components of $A$ and $B$ respectively. Now the bijection between $P_A$ and $P_B$, together with the $\subseteq$ relation, ensure that $\psi$ sends each element of $P_A$ to an element of $P_B$. Thus $\psi$ is a homeomorphism between $A$ and $B$.

Next we sketch a proof that $\mathsf{R}_1$ is Borel reducible to the graph isomorphism. We again define a countable structure as a complete invariant. Given a closed subset $A\subseteq[0,1]$, we define a structure
$$ \mathfrak{T}_A=\{ V_A, U_A, <\} $$
where $V_A$ is the set of all maximal open intervals contained in the complement of $A$ in $[0,1]$, $U_A$ is the set of all maximal open intervals contained in $A$, and $<$ compares all intervals in $U_A\cup V_A$ in their natural order. Formally, our language $L'$ consists of unary relation symbols $U$ and $V$ and a binary relation symbol $<$, and the $L'$-structures we consider satisfy the the following axiom in addition to the axioms of linear order for $<$:
\begin{itemize}
\item $\forall x\ (V(x)\vee U(x))\wedge\neg (V(x)\wedge U(x))$
\end{itemize}
We claim that for closed subsets $A, B\subseteq [0,1]$, there is an order-preserving homeomorphism $f:[0,1]\to [0,1]$ with $f[A]=B$ iff $\mathfrak{T}_A, \mathfrak{T}_B$ are isomorphic. First, if there is an order-preserving homeomorphism $f:[0,1]\to [0,1]$ with $f[A]=B$, then $f[V_A]=V_B$, $f[U_A]=U_B$, and $f$ preserves the order $<$ for elements of $V_A\cup U_A$. Thus $f$ induces an isomorphism from $\mathfrak{T}_A$ to $\mathfrak{T}_B$. Conversely, if $\varphi$ is an isomorphism from $\mathfrak{T}_A$ to $\mathfrak{T}_B$, then $\varphi$ induces an order-preserving homeomorphism on $[0,1]$ that sends $\partial A=[0,1]-\bigcup (U_A\cup V_A)$ to $\partial B=[0,1]-\bigcup (U_B\cup V_B)$. Since $\varphi$ also sends $V_A$ and $V_B$, this homeomorphism sends $A$ to $B$.

To deal with the orientation of the homeomorphism we modify the construction of the countable structure as follows. Given a closed subset $A\subseteq [0,1]$, we let $A^*=\{1-x\,:\, x\in A\}$ and
$$ \mathfrak{M}_A=\{ \mathfrak{T}_A, \mathfrak{T}_{A^*}\}. $$
That is, $\mathfrak{M}_A$ is essentially an unordered pair of countable structures that encodes both $A$ and its order-reversing copy $A^*$. It is obvious that for closed $A, B\subseteq [0,1]$, $(A, B)\in \mathsf{R}_1$ iff $\mathfrak{M}_A, \mathfrak{M}_B$ are isomorphic. Formally, we encode an unordered pair $\{ S, T\}$ by $(S, T)$, with a semidirect product $\mathbb{Z}_2\ltimes S_\infty^2$ acting on the space of an ordered pair of structures. Since $\mathbb{Z}_2\ltimes S_\infty^2$ is topologically isomorphic to a closed subgroup of $S_\infty$, it follows from Theorem~\ref{gi} that the orbit equivalence relation is Borel reducible to the graph isomorphism.

Finally we show that the graph isomorphism is Borel reducible to $\mathsf{R}_1$. For this we will actually assign to each countable linear ordering $R$ a zero-dimensional closed subset $A_R\subseteq [0,1]$ as complete invariant. The objective is to define $A_R$ so that $\mathfrak{T}_{A_R}$ from the construction above will be isomorphic to $R$. Then $R\mapsto A_R$ will be a Borel reduction from the isomorphism relation of all linear orderings to $\mathsf{R}_1$, and by Theorem~\ref{gi} this gives a Borel reduction from the graph isomorphism to $\mathsf{R}_1$. Without loss of generality, assume $R$ is infinite. To construct $A_R$, first enumerate the elements of $R$ non-repeatedly as $x_n$ for $n\geq 1$. Inductively define an open interval $I_n=(a_n, b_n)\subseteq [0,1]$ as follows. Let
\[
I_{1}=\begin{cases}(0,\frac{1}{3}), & \mbox{ if $x_{1}$ is the least element,}\\ (\frac{2}{3},1), & \mbox{ if $x_{1}$ is the largest element,}\\ (\frac{1}{3},\frac{2}{3}), & \mbox{ otherwise.}\end{cases}
\]
Assume all $I_i=(a_i,b_i)$ for $i<n$ have been defined. If $x_i$ is the greatest among $\{x_1, \dots, x_{n-1}\}$ with $x_i<x_n$, and $x_j$ is the least among $\{x_1,\dots, x_{n-1}\}$ with $x_n<x_j$, then we let 
$$ a_n=\begin{cases} b_i, & \mbox{ if there is no $x\in R$ with $x_i<x<x_n$,} \\ \frac{2}{3} b_i+\frac{1}{3}a_j, & \mbox{ otherwise,}\end{cases}
$$
and
$$ b_n=\begin{cases} a_j, & \mbox{ if there is no $x\in R$ with $x_n<x<x_j$,} \\ \frac{1}{3} b_i+\frac{2}{3}a_j, & \mbox{ otherwise.}\end{cases}
$$

If $x_i$ does not exist, then we let
$$ a_n=\begin{cases} 0, & \mbox{ if $x_n\in R$ is the least element,} \\ \frac{1}{3}a_j, & \mbox{ otherwise.}\end{cases} $$
\[
b_{n}=\begin{cases}a_{j}, & \mbox{ if there is no $x\in R$ with $x_{n}<x<x_{j}$} \\ \frac{2}{3}a_{j}, & \mbox{ otherwise.}\end{cases}
\]
Similarly, if $x_j$ does not exist, then let
\[
a_{n}=\begin{cases}b_{i}, & \mbox{ if there is no $x\in R$ with $x_{i}<x<x_{n}$,}\\ \frac{2}{3}b_{i}+\frac{1}{3}, & \mbox{ otherwise.}\end{cases}
\]
$$ b_n=\begin{cases} 1, & \mbox{ if $x_n\in R$ is the largest element,} \\ \frac{1}{3}b_i+\frac{2}{3}, & \mbox{ otherwise.}\end{cases}$$
Eventually, let $A_R=[0,1]-\bigcup_{n\geq 1}I_n$. Each interval $I_n$ is a maximal open interval in the complement of $A_R$. Our construction guarantees that $A_R$ has empty interior, and so it is zero-dimensional. 
\end{proof}

\subsection{Reducing turblence into $\mathsf{C}_2$ and $\mathsf{R}_2$} It follows from results in the previous subsections that the graph isomorphism is Borel reducible to all $\mathsf{C}_n$, $\mathsf{H}_n$, and $\mathsf{R}_n$. In this final subsection we show that $\mathsf{C}_n$, $\mathsf{H}_n$, and $\mathsf{R}_n$ are not Borel reducible to the graph isomorphism. This means that these problems are strictly more complex than the graph isomorphism.

In \cite{H}, Hjorth developed a theory of turbulence for exactly this type of question. He defined a notion of turbulent actions and showed that if an action of a Polish group is turbulent, then the orbit equivalence relation is not Borel reducible to the graph isomorphism (or to the isomorphism of countable structures). He gave an example of a homeomorphism problem of compact metric spaces which is not Borel reducible to the graph isomorphism. Unfortunately, his examples are infinite-dimensional. In the following we will adapt Hjorth's construction to create 2-dimensional continua. This will show the following main result.

\begin{theorem}\label{C2not} $\mathsf{C}_2$ is not Borel reducible to the graph isomorphism.
\end{theorem}

Since $\mathsf{C}_2\leq_B \mathsf{H}_2$, the same conclusion holds for $\mathsf{H}_2$. It will be obvious from our construction that it can be used to obtain the same conclusion for $\mathsf{R}_2$. The rest of this subsection is devoted to a proof of Theorem~\ref{C2not}.

Let $G=\mathbb{Z}^\mathbb{N}$. $G$ is a Polish group under the product topology and the product group structure. Let 
$G_0=\{\vec{x}=(x_n)\in G: x_{n}/n\rightarrow 0\}$. $G_0$ is a subgroup of $G$. We equip $G_0$ with a topological structure given by the complete metric:
\[
d(\vec{x},\vec{y})=\sup_{n}|(x_{n}-y_{n})/n|.
\]
Then $G_0$ becomes a Polish group. Consider the action of $G_0$ on $G$ by translation $+$:
\[
\vec{g}\cdot \vec{x}=(g_{n}) + (x_{n})= (g_{n}+x_{n})
\]
for $\vec{g}=(g_{n})\in G_0$ and $\vec{x}=(x_{n})\in G$. The equivalence classes of the orbit equivalence relation are exactly the cosets of $G_0$ in $G$.

\begin{lemma}[\cite{H}]
The action of $G_0$ on $G$ is turbulent. Consequently, the coset equivalence relation of $G_0$ on $G$ is not Borel reducible to the graph isomorphism.
\end{lemma}

To complete the proof it suffices to show that the coset equivalence relation of $G_0$ on $G$ is Borel reducible to $\mathsf{C}_2$. For notational simplicity we will be working with $[-1,1]\times [0,1]$ rather than $[0,1]^{2}$.
We will define a Borel reduction map $F: G\mapsto \mathcal{C}([-1,1]\times [0,1])$ such that, for all $\vec{x}, \vec{y}\in G$, $\vec{x}-\vec{y}\in G_0$ iff $F(\vec{x}), F(\vec{y})$ are homeomorphic.

We first describe a preliminary construction and fix some notation. We define closed rectangles $R_{n,k}$  inside $[0,1]^2$ for $n\geq 1$ and $k\in \mathbb{Z}$. Fix an order-preserving homeomorphism $f:\mathbb{R}\rightarrow (0,1)$ so that $f(0)=\frac{1}{2}$, then $R_{n,k}$ is the rectangle with the vertices
\[
\left(\frac{1}{2n+1},f(\frac{k+1}{n})\right), \hspace{9mm} \left(\frac{1}{2n}, f(\frac{k+1}{n})\right),
\]
\[
\left(\frac{1}{2n+1},f(\frac{k}{n})\right),\mbox{ and } \left(\frac{1}{2n},f(\frac{k}{n})\right).
\]
Figure~\ref{P:1} illustrates this construction.

\begin{figure}[ht]
\centering
    \begin{tikzpicture}[scale=7.5] 
    \foreach \i in {1,...,8}
    \foreach  \j  in {0,...,4} 
    {
       \pgfmathsetmacro{\jm}{1/(1+2^(-\j/\i+1))}
       \pgfmathsetmacro{\jn}{1/(1+2^(-\j/\i))}
       \pgfmathsetmacro{\jx}{1/(2*\i+1)}
       \pgfmathsetmacro{\jy}{1/(2*\i)}
       \draw (\jx,\jm) rectangle (\jy,\jn);
       \draw (\jx,-\jm+1) rectangle (\jy,-\jn+1);
       \draw (\jx,0)--(\jx,1);
       \draw (\jy,0)--(\jy,1);
    }
    \draw [dashed] (-0.7,0)--(0.9,0)--(0.9,1)--(-0.7,1)--(-0.7,0);
    \draw (0.04,0)--(0.04,1);
    \draw [fill=red] (0.04,0.5) circle [radius=0.005];
    \node at (-0.01,0.5) [scale=0.65] {$(0,\frac{1}{2})$};
    
    \node at (0.4167,0.5833) [scale=0.8] {$R_{1,0}$};
    \node at (0.4167,0.4167) [scale=0.8] {$R_{1,-1}$};
    \node at (0.4167,0.7333) [scale=0.8] {$R_{1,1}$};
    \node at (0.4167,0.2667) [scale=0.8] {$R_{1,-2}$};
    \node at (0.4167,0.8444) [scale=0.8] {$R_{1,2}$};
    \node at (0.4167,0.1556) [scale=0.8] {$R_{1,-3}$};
    
    \node at (0.225,0.5429) [scale=0.4] {$R_{2,0}$};
    \node at (0.225,0.4571) [scale=0.4] {$R_{2,-1}$};
    \node at (0.225,0.6262) [scale=0.4] {$R_{2,1}$};
    \node at (0.225,0.3738) [scale=0.4] {$R_{2,-2}$};
  \end{tikzpicture}
  \caption{The rectangles $R_{n,k}$ for $n\geq 1$ and $k\in\mathbb{Z}$.}
  \label{P:1}
  \end{figure}
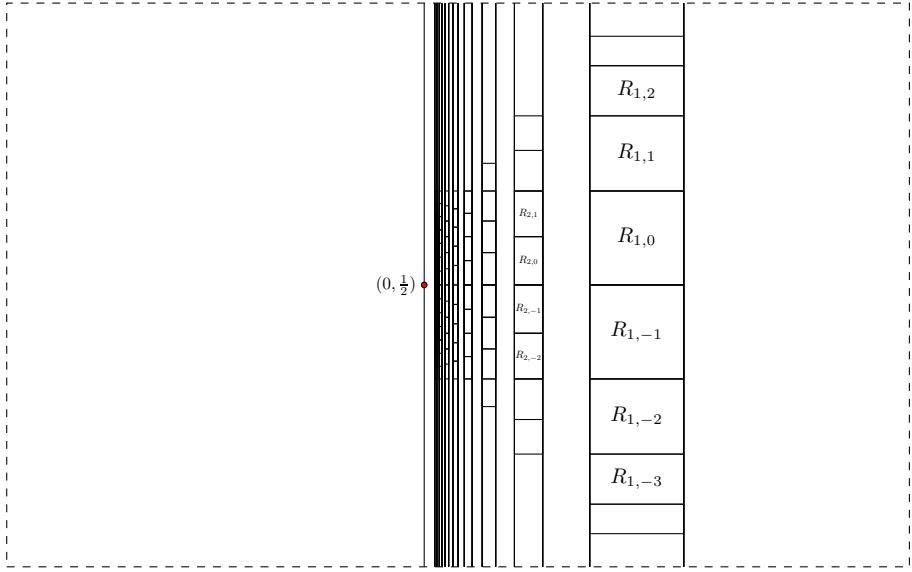
  
We use $\partial R_{n,k}$ and $R_{n,k}^o$ to denote the boundary and the interior of $R_{n,k}$, respectively. 

For any $n\geq 1$ and $k,l\in\mathbb{Z}$, define a homeomorphism $\sigma_{n,k,l}: R_{n,k}\rightarrow R_{n,k+l}$ by
$\sigma_{n,k,l}(a,f(b))=(a,f(b+l/n))$ for $a\in [1/(2n+1), 1/2n]$ and $b\in [k/n, (k+1)/n]$.

We are now ready to define the map $F$. Given $\vec{x}=(x_n)\in G$, let 
\[
F(\vec{x})=I_0\cup \bigcup_{n\geq 1}(I_n^{\vec{x}}\cup C^{\vec{x}}_n)
\]
where $I_{0}:=[-1,0]\times\{1/2\}\cup \{0\}\times[0,1]$, and for each $n\geq 1$,
\[
I^{\vec{x}}_n:=\mbox{the closure of } (R_{n,x_n+1}\cup\bigcup_{k\neq x_n+1} \partial  R_{n,k}),
\]
and 
\[
C^{\vec{x}}_n=\left\{\left(\frac{1}{2n+1+\lambda}, f(\frac{x_n+1/2}{n}(1-\lambda)+\frac{x_{n+1}+1/2}{n+1}\lambda)\right)\,:\,\lambda\in[0,1]\right\}.
\]
The closed set $F(\vec{x})$ consists of three parts: a T-shaped path-component $I_{0}$, a sequence of ``stripes'' $(I^{\vec{x}}_{n})$, and a sequence of curved line segments $(C^{\vec{x}}_n)$ connecting the neighboring stripes. Figure~\ref{F:2} illustrates this construction, and Figure~\ref{F:3} gives a better local view of the $n$-th and the $(n+1)$-st stripes. 

\begin{figure}[h]
\centering
\begin{tikzpicture}[scale=7.5] 
    \foreach \i in {1,...,8}
    \foreach  \j  in {0,...,4} 
    {
       \pgfmathsetmacro{\jm}{1/(1+2^(-\j/\i+1))}
       \pgfmathsetmacro{\jn}{1/(1+2^(-\j/\i))}
       \pgfmathsetmacro{\jx}{1/(2*\i+1)}
       \pgfmathsetmacro{\jy}{1/(2*\i)}
       \draw (\jx,\jm) rectangle (\jy,\jn);
       \draw (\jx,-\jm+1) rectangle (\jy,-\jn+1);
       \draw (\jx,0)--(\jx,1)--(\jy,1)--(\jy,0)--(\jx,0);
    }
    \draw [dashed] (-0.7,0)--(0.9,0)--(0.9,1)--(-0.7,1)--(-0.7,0);
    \draw (-0.7,0.5)--(0.04,0.5);
    \draw (0.04,0)--(0.04,1);
    \draw [fill=red] (0.04,0.5) circle [radius=0.005];
    
    \draw [fill=orange] (0.3333,0.8889) rectangle (0.5,0.9412);
    \draw [fill=black] (0.3333,0.8444) circle [radius=0.003];
    
    \draw [fill=orange] (0.2,0.3333) rectangle (0.25,0.4142);
    \draw [fill=black] (0.2,0.29727) circle [radius=0.003];
    \draw [fill=black] (0.25,0.29727) circle [radius=0.003];
    \draw plot [smooth, tension=1] coordinates { (0.25,0.29727) (0.3, 0.5) (0.3333,0.8444)};
    
    \draw [fill=orange] (0.14286,0.5) rectangle (0.16667,0.55751);
    \draw [fill=black] (0.14286,0.4712) circle [radius=0.003];
    \draw [fill=black] (0.16667,0.4712) circle [radius=0.003];
    \draw plot [smooth, tension=1] coordinates { (0.16667,0.4712) (0.19, 0.4) (0.2,0.29727)};

    \draw [fill=orange] (0.1111,0.6271) rectangle (0.125,0.6667);
    \draw [fill=black] (0.1111,0.60645) circle [radius=0.003];
    \draw [fill=black] (0.125,0.60645) circle [radius=0.003];
    \draw plot [smooth, tension=1] coordinates { (0.14286,0.4712) (0.138,0.55) (0.125,0.60645)};

    \draw plot [smooth,tension=1] coordinates { (0.1111,0.60645)  (0.1,0.2)};
    
  \end{tikzpicture}
  \caption{The contruction of $F(\vec{x})$.} 
  \label{F:2}
  \end{figure}
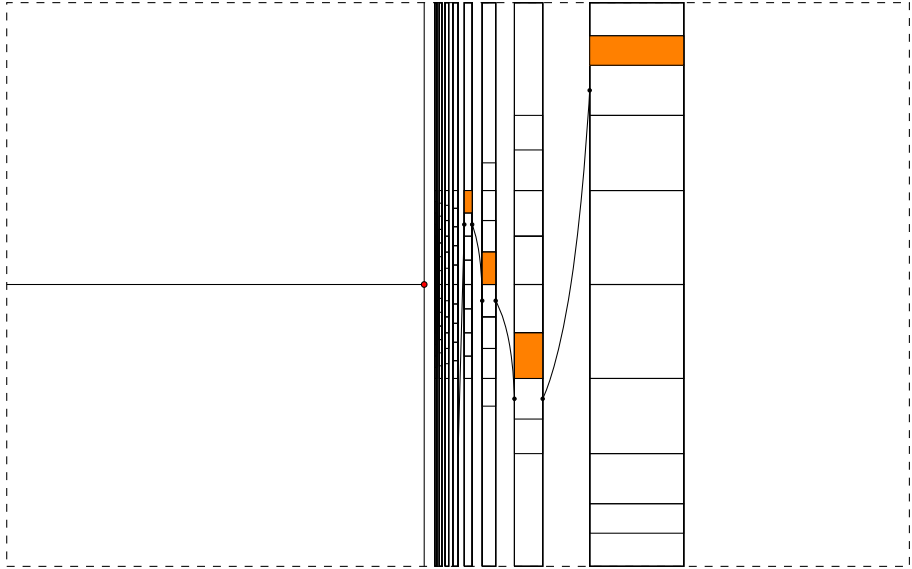

\begin{figure}[h]
\centering
\begin{tikzpicture}
  
  \draw (-2,0) rectangle (-1,1);
  \draw [fill=orange] (-2,0)  rectangle  (-1,1);
\draw (-2,1) rectangle (-1,2);
\draw (-2,2) rectangle (-1,2.7);
\draw (-2,2.7) rectangle (-1,3.4);
\draw (-2,3.4) rectangle (-1,3.4+0.49);
\draw (-2,3.4+0.49) rectangle (-1,3.4+2*0.49);

\draw (-2,4.38) rectangle (-1, 4.723);
\draw (-2,4.723) rectangle (-1, 5.066);

\draw (-2,5.066) rectangle (-1, 5.3061);
\draw (-2,5.3061) rectangle (-1, 5.5462);

\draw (-2,0) rectangle  (-1,-0.6);
\draw (-2,-0.6) rectangle (-1,-1.2);

\draw (-2,-1.2) rectangle (-1, -1.56);
\draw (-2,-1.56) rectangle (-1, -1.92);

\draw (-2, -1.92) rectangle (-1,-2.136);
\draw (-2, -2.136) rectangle (-1,-2.352);
\draw (-2, -2.352) rectangle (-1, -2.4816);
\draw (-2, -2.4816) rectangle (-1, -2.6112);

\draw (0,0) rectangle (2,2);
\node at (1,1) {$R_{n,x_n-1}$};
\draw (0,2) rectangle (2,3.4);
\node at (1,2.7) {$R_{n,x_n}$};
\draw (0,3.4) rectangle (2,3.4+2*0.49);
\draw [fill=orange] (0,3.4) rectangle (2,3.4+2*0.49);
\node at (1, 3.89) {$R_{n,x_n+1}$};
\draw (0,4.38) rectangle (2, 5.066);
\node at (1, 4.723) {$R_{n,x_n+2}$};
\draw (0,5.066) rectangle (2, 5.5462);

\draw (0,0) rectangle (2,-1.2);
\draw (0,-1.2) rectangle (2, -1.92);
\draw (0, -1.92) rectangle (2,-2.352);
\draw (0, -2.352) rectangle (2, -2.6112);

\draw [fill=black] (0,2.7) circle [radius=0.05];
\draw [fill=black] (-1,-0.3) circle [radius=0.05];

\draw plot [smooth, tension=1] coordinates { (-1,-0.3) (-0.7,1.3)(0,2.7)};
\node at (-0.45, 0.8) [scale=0.6]{$C_{n}$};
\draw [fill=black] (2,2.7) circle [radius=0.05];
\draw [fill=black] (-2,-0.3) circle [radius=0.05];

\draw plot [smooth,tension=1] coordinates {(2, 2.7) (2.4,1.5) (3, 0.7)};
\node at (2.9, 1.5) [scale=0.6] {$C_{n-1}$};

\draw plot [smooth, tension=1] coordinates {(-2,-0.3) (-2.7,0.9) (-3.5,1.4)}; 
\node at (-3.2, 0.75)[scale=0.6] {$C_{n+1}$};
\node at (1,6) {$I_{n}^{\vec{x}}$};
\node at (-1.5,6) {$I_{n+1}^{\vec{x}}$};

\end{tikzpicture}
\caption{The ``stripes'' $I_{n}^{\vec{x}}$ and $I_{n+1}^{\vec{x}}$.} 
\label{F:3}
\end{figure}
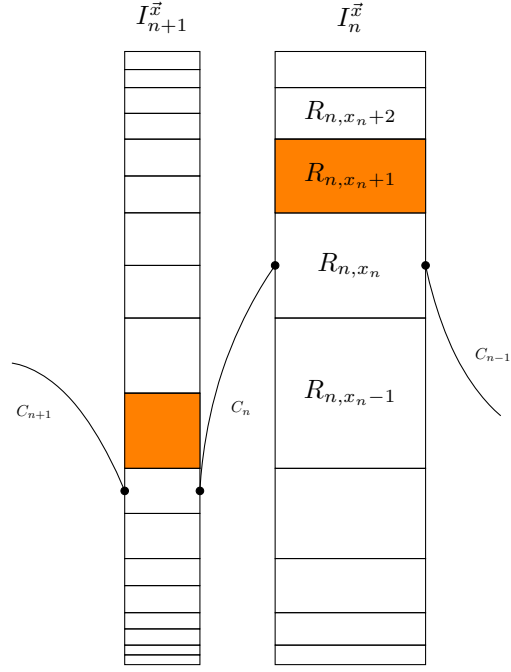

Note that $F(\vec{x})$ thus constructed is a continuum with two path-components as follows: 
\begin{itemize}
    \item $I_0^{\vec{x}}$, where there are exactly three non-cut points of its own.
    \item $\bigcup_{n\geq 1} (I_n^{\vec{x}}\cup C_n^{\vec{x}})$, where there are infinitely many non-cut points.
\end{itemize}

For one direction of the proof, suppose $\vec{x}-\vec{y}\in G_0$, i.e. $|x_n-y_n|/n\rightarrow 0$ as $n\rightarrow \infty$. We show that there exists a homeomorphism between $F(\vec{x})$ and $F(\vec{y})$. Actually, we prove a stronger result by constructing an autohomeomorphism $\varphi$ on $[-1,1]\times[0,1]$ with $\varphi(F(\vec{x}))=F(\vec{y})$. 

We define an autohomeomorphism $\sigma$ on $(0,1)^2$:
\begin{itemize}
\item On the stripes $I_{n}^{\vec{x}}$ for $n\geq 1$, we let $\sigma |_{R_{n,k}}=\sigma_{n,k,y_n-x_n}$ for all $k\in\mathbb{Z}$;
\item In the domains of the form
\[
\left(\frac{1}{2n+2},\frac{1}{2n+1}\right)\times (0,1)
\]
where $n\geq 1$, we let
\[
\sigma\left(\frac{1}{2n+1+\lambda},f(z)\right) \\
=\left(\frac{1}{2n+1+\lambda}, f(z+\frac{(y_n-x_n)(1-\lambda)}{n}+\frac{(y_{n+1}-x_{n+1})\lambda}{n+1})\right).
\]
for all $\lambda \in(0,1)$ and $z\in\mathbb{R}$;
\item In the domain $(1/2,1)\times (0,1)$, we let 
\[
\sigma\left(\frac{1}{1+\lambda},f(z)\right)=\left(\frac{1}{1+\lambda},f(z+(y_1-x_1)\lambda)\right)
\]
for all $\lambda\in(0,1)$ and $z\in\mathbb{R}$.
\end{itemize}
Note that $\sigma[I^{\vec{x}}_n]=I^{\vec{y}}_n$, $\sigma[C^{\vec{x}}_n]=C^{\vec{y}}_n$, and that $\sigma$ is continuous everywhere in $(0,1)^2$. By the assumption that $|x_n-y_n|/n\rightarrow 0$ as $n\rightarrow \infty$, we have that for $p\in (0,1)^2$, 
\[
d(p,\sigma(p))\rightarrow 0
\]
as $p\rightarrow \partial [0,1]^2$. Hence the above $\sigma$ uniquely extends to an autohomeomorphism $\varphi$ of $[0,1]^2$ so that $\varphi(p)=p$ for $p\in\partial [0,1]^2$. Now further extending this $\varphi$ by an identity on $[-1,0]\times [0,1]$ (and in particular an identity on $I_0$), we obtain an autohomeomorphism of $[-1,1]\times[0,1]$ with $\varphi[F(\vec{x})]=F(\vec{y})$.

For the converse direction, suppose $\pi: F(\vec{x})\to F(\vec{y})$ is a homeomorphism. We want to show that $|x_n-y_n|/n\rightarrow 0$ as $n\rightarrow \infty$. Since $\pi$ maps each path-component of $F(\vec{x})$ to a path-component of $F(\vec{y})$, we have $\pi[I_0^{\vec{x}}]=I_0^{\vec{y}}$ and
\[\pi \left[\bigcup_{n\geq 1} (I_n^{\vec{x}}\cup C_n^{\vec{x}})\right]=\bigcup_{n\geq 1} (I_n^{\vec{y}}\cup C_n^{\vec{y}}).
\]

Next, we restrict our spaces to $\bigcup (I_n^{\vec{x}}\cup C_n^{\vec{x}})$ and $\bigcup (I_n^{\vec{y}}\cup C_n^{\vec{y}})$, respectively, in order to show that each $I_n^{\vec{x}}$ must be mapped to $I_n^{\vec{y}}$, and each $C_n^{\vec{x}}$ must be mapped to $C_n^{\vec{y}}$. 
\begin{claim}
For all $n\geq 1$, $\pi(I_{n}^{\vec{x}})=I_{n}^{\vec{y}}$ and $\pi(C_n^{\vec{x}})=C_n^{\vec{y}}$.
\end{claim}
\begin{proof}[Proof of Claim:]
Note that for any $\vec{u}\in G$, $\bigcup C_n^{\vec{u}}$ are exactly the set of all cut-points in $\bigcup (I_n^{\vec{u}}\cup C_n^{\vec{u}})$. Therefore we must have
\[
\pi\left[\bigcup C_n^{\vec{x}}\right]= \bigcup C_n^{\vec{y}}.
\]
Note that for each $\vec{u}\in G$ and $n\geq 1$, $C_n^{\vec{u}}$ is a path-component of $\bigcup C_n^{\vec{u}}$. In fact, $C_n^{\vec{u}}$ can be topologically characterized as the unique path-component $C$ of $\bigcup C_n^{\vec{u}}$ so that $F(\vec{u})-C$ contains a path-component $D$ such that the set of all cut-points of $D$ has exactly $n-1$ many path-components. It follows that for all $n\geq 1$, $\pi[C^{\vec{x}}_n]=C^{\vec{y}}_n$.

Now each $I^{\vec{u}}_n$ can be topologically characterized inductively as follows. $I^{\vec{u}}_1$ is the unique path-component of $F(\vec{u})-(C^{\vec{u}}_1)^o$ without cut-points. For $n>1$, $I^{\vec{u}}_n$ is the unique path-component of 
$$F(\vec{u})-\left(C^{\vec{u}}_n\cup \bigcup_{i<n}(C_i^{\vec{u}}\cup I^{\vec{u}}_i)\right)^o $$
without cut-points. Thus for all $n\geq 1$, $\pi[C^{\vec{x}}_n]=C^{\vec{y}}_n$.
\end{proof}


\begin{claim}
For all $n\geq 1$ and $k\in \mathbb{Z}$, $\pi[\partial R_{n,k}]=\partial R_{n,k+y_n-x_n}$ and $\pi[R_{n,x_n+1}]=R_{n,y_n+1}$.
\end{claim}

\begin{proof}[Proof of Claim:] We only show the case when $n=1$. The case $n\geq 2$ is proved with the same argument. By the last claim, we know $\pi[I^{\vec{x}}_1]=I^{\vec{y}}_1$ and $\pi[C^{\vec{x}}_1]=C^{\vec{y}}_1$. However, note that $I^{\vec{x}}_1$ intersects $C^{\vec{x}}_1$ at a unique point, namely $p^{\vec{x}}=(1/3, f(x_1+1/2))$. Similarly, $I^{\vec{y}}_1\cap C^{\vec{y}}_1=\{p^{\vec{y}}\}$, where $p^{\vec{y}}=(1/3, f(y_1+1/2))$. This implies that $\pi(p^{\vec{x}})=p^{\vec{y}}$.

Before continuing, we introduce some additional notation. We think of the boundary of $R_{1,k}$ being divided into four parts: the ``left'' side will be denoted by $l_{k}$, the ``right'' side by $r_{k}$, and the ``bottom'' side by $b_{k}$. See Figure \ref{F:4}. With these, the ``top'' side of the boundary of $R_{1,k}$ is $b_{k+1}$.  $p^{\vec{x}}$ is on the side $l_{x_1}$.
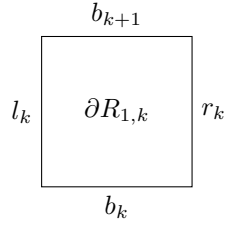
\begin{figure}[h]
\centering
\begin{tikzpicture}
\draw (0,0) rectangle node {$\partial R_{1,k}$} (2,2);
\node at (1,2) [above] {$b_{k+1}$};
\node at (0,1) [left] {$l_{k}$};
\node at (2,1) [right] {$r_{k}$};
\node at (1,0) [below] {$b_{k}$};
\end{tikzpicture}
\caption{The boundaries of the rectangle $R_{1,k}$}
\label{F:4}
\end{figure}

Now the set of all cut-points of $I^{\vec{x}}_1-\{p^{\vec{x}}\}$ consists of exactly $l_{x_1}\cup r_{x_1}$, and each of $l_{x_1}$ and $r_{x_1}$ is a path-component of $l_{x_1}\cup r_{x_1}$. Similarly, $p^{\vec{x}}$ is on the side $l_{y_1}$, while $l_{y_1}, r_{y_1}$ are the two path-components of the set of all cut-points of $I^{\vec{y}}_1-\{p^{\vec{y}}\}$. This implies that $\pi[l_{x_1}]=l_{y_1}$ and $\pi[r_{x_1}]=r_{y_1}$. 

Note that $I^{\vec{x}}_1-(l_{x_1}\cup r_{x_1})$ contains exactly three components:
\begin{itemize}
\item $b_{x_1}^o$, which contains only cut-points;
\item $R_{1, x_1+1}\cup\bigcup_{k>x_1+1}\partial R_{1, k}$, which contains only non-cut points;
\item $\bigcup_{k<x_1}\partial R_{1,k}-b_{x_1}$, which contains both cut-points and non-cut points.
\end{itemize}
Moreover, $R_{1, x_1+1}$ consists of exactly the points $p$ in $I^{\vec{x}}_1$ such that any neighborhood of $p$ contains a homeomorphic copy of the upper half plane $\mathbb{R}\times [0,+\infty)$. 

All of this analysis can be done similarly on the $\vec{y}$ side. It follows that we must have $\pi[b_{x_1}]=b_{y_1}$, $\pi[R_{1, x_1+1}]=R_{1, y_1+1}$,
$$\pi[\bigcup_{k>x_1+1}\partial R_{1,k}-b_{x_1+2}]=\bigcup_{k>y_1+1}\partial R_{1,k}-b_{y_1+2},$$ and $$\pi[\bigcup_{k<x_1}\partial R_{1,k}-b_{x_1}]=\bigcup_{k<y_1}\partial R_{1,k}-b_{y_1}.$$

Note that 
$$\bigcup_{k>x_1+1}\partial R_{1,k}-b_{x_1+2}=l_{x_1+2}^o\cup r_{x_1+2}^o\cup \bigcup_{k>x_1+2} \partial R_{1,k},$$
and $l_{x_1+2}$ and $r_{x_1+2}$ are the two path-components of the set of all cut-points of the above set. From this we get $\pi[l_{x_1+2}\cup r_{x_1+2}]=l_{y_1+2}\cup r_{y_1+2}$, $\pi[b_{x_1+3}]=b_{y_1+3}$ and 
$$ \pi[\bigcup_{k>x_1+2}\partial R_{1,k}-b_{x_1+3}]=\bigcup_{k>y_1+2}\partial R_{1,k}-b_{y_1+3}.$$
A repetition of the argument shows that $\pi[\partial  R_{1,k}]=\partial R_{1, k+y_1-x_1}$ for all $k>x_1+1$.

A similar argument shows that $\pi[\partial R_{1,k}]=\partial R_{1,k+y_1-x_1}$ for all $k<x_1$. The claim is thus proved. 
\end{proof}

Finally, we look back at the path-component $I_0$ in $F(\vec{x})$ and in $F(\vec{y})$. We have $\pi[I_0]=I_0$. Notice that $(0,1/2)$ is a distinguished by the topological property that it is the unique cut-point in $I_{0}$ so that removing it will result in three path-components. Therefore,  $\pi$ fixes the point $(0,1/2)$. From Claim 2 above, we have $\pi(\partial R_{n,0}^{\vec{x}})=\partial R_{n,y_n-x_n}^{\vec{y}}$ for all $n\geq 1$. As $n\rightarrow \infty$, $\partial R_{n,0}^{\vec{x}}$ converges to the fixed point $(0,1/2)$, so we must have that $\partial R_{n, y_n-x_n}^{\vec{y}}$ converges also to $(0,1/2)$. This implies that $|y_n-x_n|/n\rightarrow 0$.

\end{document}